\documentclass{amsart}
\usepackage{amssymb, amsmath, mathrsfs, verbatim, url}

\theoremstyle{plain}
\newtheorem{theorem}{Theorem}[section]
\newtheorem{lemma}[theorem]{Lemma}
\newtheorem{proposition}[theorem]{Proposition}
\newtheorem{corollary}[theorem]{Corollary}

\theoremstyle{definition}
\newtheorem{definition}[theorem]{Definition}
\newtheorem{question}[theorem]{Question}

\newcommand{\mini}{\mathsf{min}}
\newcommand{\mesh}{\mathsf{mesh}}
\newcommand{\diam}{\mathsf{diam}}
\newcommand{\di}{\mathsf{d}}
\newcommand{\Bi}{\mathsf{B}}
\newcommand{\id}{\mathsf{id}}

\newcommand{\cl}{\mathsf{cl}}

\newcommand{\Gd}{\mathsf{G_\delta}}
\newcommand{\Fs}{\mathsf{F_\sigma}}
\newcommand{\KR}{\mathsf{KR}}
\newcommand{\CDH}{\mathsf{CDH}}
\newcommand{\CB}{\mathsf{CB}}
\newcommand{\CH}{\mathsf{CH}}
\newcommand{\ZFC}{\mathsf{ZFC}}
\newcommand{\Ell}{\mathsf{L}}
\newcommand{\MA}{\mathsf{MA}}

\newcommand{\MAc}{\mathsf{MA}(\mathrm{countable})}
\newcommand{\SLH}{\mathsf{SLH}}
\newcommand{\Aa}{\mathcal{A}}
\newcommand{\BB}{\mathcal{B}}

\newcommand{\FF}{\mathcal{F}}

\newcommand{\KK}{\mathcal{K}}

\newcommand{\UU}{\mathcal{U}}
\newcommand{\VV}{\mathcal{V}}
\newcommand{\WW}{\mathcal{W}}

\newcommand{\RRR}{\mathbb{R}}
\newcommand{\QQQ}{\mathbb{Q}}

\begin{document}

\title[Countable dense homogeneity in powers]{Countable dense homogeneity in powers of zero-dimensional definable spaces}

\author{Andrea Medini}
\address{Kurt G\"odel Research Center for Mathematical Logic
\newline\indent University of Vienna
\newline\indent W\"ahringer Stra{\ss}e 25
\newline\indent A-1090 Wien, Austria}
\email{andrea.medini@univie.ac.at}
\urladdr{http://www.logic.univie.ac.at/\~{}medinia2/}

\date{April 22, 2015}

\thanks{The author acknowledges the support of the FWF grant I 1209-N25.}

\begin{abstract}
We show that, for a coanalytic subspace $X$ of $2^\omega$, the countable dense homogeneity of $X^\omega$ is equivalent to $X$ being Polish. This strengthens a result of Hru\v{s}\'ak and Zamora Avil\'es. Then, inspired by results of Hern\'andez-Guti\'errez, Hru\v{s}\'ak and van Mill, using a technique of Medvedev, we construct a non-Polish subspace $X$ of $2^\omega$ such that $X^\omega$ is countable dense homogeneous. This gives the first $\ZFC$ answer to a question of Hru\v{s}\'ak and Zamora Avil\'es. Furthermore, since our example is consistently analytic, the equivalence result mentioned above is sharp. Our results also answer a question of Medini and Milovich. Finally, we show that if every countable subset of a zero-dimensional separable metrizable space $X$ is included in a Polish subspace of $X$ then $X^\omega$ is countable dense homogeneous.
\end{abstract}

\maketitle

\section{Introduction}

As is common in the literature about countable dense homogeneity, by \emph{space} we will always mean ``separable metrizable topological space''. By \emph{countable} we will always mean ``at most countable''. Our reference for general topology is \cite{vanmilli}. Our reference for descriptive set theory is \cite{kechris}. For all other set-theoretic notions, we refer to \cite{kunen}. Recall the following definitions. A space is \emph{Polish} if it admits a complete metric. A subspace of a Polish space is \emph{analytic} if it is the continuous image of a Polish space, and it is \emph{coanalytic} if its complement is analytic. A space $X$ is \emph{countable dense homogeneous} (briefly, $\CDH$) if for every pair $(A,B)$ of countable dense subsets of $X$ there exists a homeomorphism $h:X\longrightarrow X$ such that $h[A]=B$.

The fundamental positive result in the theory of $\CDH$ spaces is the following (see \cite[Theorem 5.2]{andersoncurtisvanmill}). In particular, it shows that the Cantor set $2^\omega$, the Baire space $\omega^\omega$, the Euclidean spaces $\RRR^n$, the spheres $S^n$ and the Hilbert cube $[0,1]^\omega$ are all examples of $\CDH$ spaces. See \cite[Sections 14-16]{arkhangelskiivanmill} for much more on this topic. Recall that a space $X$ is \emph{strongly locally homogeneous} (briefly, $\SLH$) if there exists a base $\BB$ for $X$ such that for every $U\in\BB$ and $x,y\in U$ there exists a homeomorphism $h:X\longrightarrow X$ such that $h(x)=y$ and $h\upharpoonright (X\setminus U)=\id_{X\setminus U}$.

\begin{theorem}[Anderson, Curtis, van Mill]\label{maincdh}
Every Polish $\SLH$ space is $\CDH$.
\end{theorem}

This article is ultimately motivated by the second part of the following question (see \cite{fitzpatrickzhoud}), which is Problem 387 from the book ``Open problems in topology''. Recall that a space $X$ is \emph{homogeneous} if for every pair $(x,y)$ of elements of $X$ there exists a homeomorphism $h:X\longrightarrow X$ such that $h(x)=y$.
\begin{question}[Fitzpatrick, Zhou]\label{mainquestion}
Which subspaces $X$ of $2^\omega$ are such that $X^\omega$ is homogeneous? $\CDH$?
\end{question}
While the first question was answered by the following remarkable result\footnote{Subsequently, Theorem \ref{hompower} was greatly generalized by Dow and Pearl (see \cite[Theorem 2]{dowpearl}), by combining the methods of Lawrence with the technique of elementary submodels.} (see \cite[page 3057]{lawrence}), the second question is still open.
\begin{theorem}[Lawrence]\label{hompower}
Let $X$ be a subspace of $2^\omega$. Then $X^\omega$ is homogeneous.
\end{theorem}

However, if one focuses on \emph{definable} spaces, it is possible to obtain the following result (see \cite[Corollary 2.4]{hrusakzamoraaviles}).
\begin{theorem}[Hru\v{s}\'ak, Zamora Avil\'es]\label{borelcdh}
Let $X$ be a Borel subspace $2^\omega$. If $X$ is $\CDH$ then $X$ is Polish.
\end{theorem}
Furthermore, there exist consistent examples of an analytic subspace of $2^\omega$ and a coanalytic subspace of $2^\omega$ that are $\CDH$ but not Polish (see \cite[Theorem 2.6]{hrusakzamoraaviles}), which show that Theorem \ref{borelcdh} is sharp. Such definable examples could not have been constructed in $\ZFC$ because, under the axiom of Projective Determinacy, Theorem \ref{borelcdh} extends to all projective subspaces of $2^\omega$ (see \cite[Corollary 2.7]{hrusakzamoraaviles}).

Using Theorem \ref{borelcdh} (see also the proof of Theorem \ref{coanalyticcdhpower}), it is possible to obtain the following result (see \cite[Theorem 3.2]{hrusakzamoraaviles}), which was the first breakthrough on the second part of Question \ref{mainquestion}.
\begin{theorem}[Hru\v{s}\'ak, Zamora Avil\'es]\label{borelcdhpower}
Let $X$ be a Borel subspace of $2^\omega$. Then the following are equivalent.
\begin{itemize}
\item $X$ is Polish.
\item $X^\omega$ is $\CDH$.
\end{itemize}
\end{theorem}
As above, it is easy to realize that, under the axiom of Projective Determinacy, Theorem \ref{borelcdhpower} extends to all projective subspaces of $2^\omega$.

At this point, it seems natural to wonder whether the ``Borel'' assumption in the above theorem can be dropped. In other words, is being Polish the characterization that we are looking for? This is precisely what the following question asks (see \cite[Question 3.2]{hrusakzamoraaviles}).
\begin{question}[Hru\v{s}\'ak, Zamora Avil\'es]\label{Gdeltaquestion}
Is there a non-Polish subspace $X$ of $2^\omega$ such that $X^\omega$ is $\CDH$?
\end{question}

The following (see \cite[Theorem 21]{medinimilovich}) is the first consistent answer\footnote{Subsequently, Hern\'andez-Guti\'errez and Hru\v{s}\'ak showed that both $\FF$ and $\FF^\omega$ are $\CDH$ whenever $\FF$ is a non-meager $\mathsf{P}$-filter on $\omega$ (see \cite[Theorem 1.6]{hernandezgutierrezhrusak}). In fact, as it was recently shown by Kunen, Medini and Zdomskyy, a filter on $\omega$ is $\CDH$ if and only if it is a non-meager $\mathsf{P}$-filter (see \cite[Theorem 10]{kunenmedinizdomskyy}). However, it is a long-standing open problem whether non-meager $\mathsf{P}$-filters exist in $\ZFC$ (see \cite{justmathiasprikrysimon} or \cite[Section 4.4.C]{bartoszynskijudah}).} to the above question, where ultrafilters on $\omega$ are viewed as subspaces of $2^\omega$ through characteristic functions.
\begin{theorem}[Medini, Milovich]\label{consistentanswer}
Assume that $\MAc$ holds. Then there exists a non-principal ultrafilter $\UU$ on $\omega$ such that $\UU^\omega$ is $\CDH$.
\end{theorem}
Since a non-principal ultrafilter on $\omega$ can never be analytic or coanalytic (see \cite[Section 2]{medinimilovich}), the following question seems natural (see \cite[Question 6]{medinimilovich}).
\begin{question}[Medini, Milovich]\label{optimizedefinability}
Is there a non-Polish analytic subspace $X$ of $2^\omega$ such that $X^\omega$ is $\CDH$? Coanalytic? 
\end{question}

We will give a stronger version of Theorem \ref{borelcdhpower} (namely, Theorem \ref{coanalyticcdhpower}) and show that this version is sharp (see Theorem \ref{mainexample}), while simultaneously answering Question \ref{Gdeltaquestion} and Question \ref{optimizedefinability}. The countable dense homogeneity of the example given by Theorem \ref{mainexample} will follow from Theorem \ref{main}, whose proof uses the technique of Knaster-Reichbach covers. Finally, by combining Theorem \ref{main} with several results about $\omega$-th powers, we will obtain a simple sufficient condition for the countable dense homogeneity of $X^\omega$ (see Theorem \ref{sufficient}).

\section{Some preliminary notions}

Recall that a space is \emph{crowded} if it is non-empty and it has no isolated points. Given spaces $X$ and $Y$, we will write $X\approx Y$ to mean that $X$ and $Y$ are homeomorphic. Given a space $Z$, we will say that a subspace $S$ of $Z$ is a \emph{copy} of a space $X$ if $S\approx X$. The following four classical results are used freely throughout this entire article (see \cite[Theorem 1.5.5]{vanmilli} and \cite[Theorem 1.9.8 and Corollary 1.9.9]{vanmilli}, \cite[Theorem A.6.3]{vanmilli}, \cite[Theorem 13.6]{kechris} and \cite[Lemma A.6.2]{vanmilli} respectively).
\begin{theorem}
Let $X$ be a zero-dimensional space.
\begin{itemize}
\item If $X$ is compact and crowded then $X\approx 2^\omega$.
\item If $X$ is Polish and nowhere locally compact then $X\approx\omega^\omega$.
\end{itemize}
\end{theorem}
\begin{theorem}
Let $X$ be a subspace of a Polish space $Z$. Then $X$ is Polish if and only if $X$ is a $\Gd$ subset of $Z$.
\end{theorem}
\begin{theorem}
Let $Z$ be a Polish space. If $X$ is an uncountable Borel subspace of $Z$ then $X$ contains a copy of $2^\omega$.
\end{theorem}
\begin{proposition}
Let $I$ be a countable set. If $X_i$ is Polish for every $i\in I$ then $\prod_{i\in I}X_i$ is Polish.
\end{proposition}

Recall that a space $X$ is \emph{completely Baire} (briefly, $\CB$) if every closed subspace of $X$ is a Baire space. For a proof of the following result, see \cite[Corollary 21.21]{kechris} and \cite[Corollary 1.9.13]{vanmilli}.
\begin{theorem}[Hurewicz]\label{hurewicz}
Let $X$ be a space. Consider the following conditions.
\begin{enumerate}
\item\label{polishcb} $X$ is Polish.
\item\label{cb} $X$ is $\CB$.
\item\label{noQ} $X$ does not contain any closed copy of $\QQQ$.
\end{enumerate}
The implications $(\ref{polishcb})\rightarrow(\ref{cb})\leftrightarrow (\ref{noQ})$ hold for every $X$. If $X$ is a coanalytic subspace of some Polish space then the implication $(\ref{polishcb}) \leftarrow (\ref{cb})$ holds as well.
\end{theorem}

Recall that a \emph{$\lambda$-set} is a space in which every countable set is $\Gd$. Observe that no $\lambda$-set can contain a copy of $2^\omega$. Recall that a \emph{$\lambda'$-set} is a subspace $X$ of $2^\omega$ such that $X\cup D$ is a $\lambda$-set for every countable $D\subseteq 2^\omega$. For a proof of Lemma \ref{unionlambda}, see \cite[Theorem 7.2]{millers}. For a proof of Theorem \ref{zfclambda}, which is based on the existence of a Hausdorff gap, see \cite[Theorem 5.5]{millers} and the argument that follows it.
\begin{lemma}[Sierpi\'{n}ski]\label{unionlambda}
A countable union of $\lambda'$-sets is a $\lambda'$-set.
\end{lemma}
\begin{theorem}[Sierpi\'{n}ski]\label{zfclambda}
There exists a $\lambda'$-set of size $\omega_1$.
\end{theorem}

Recall that a subspace $B$ of an uncountable Polish space $Z$ is a \emph{Bernstein set} if $B\cap K\neq\varnothing$ and $(Z\setminus
B)\cap K\neq\varnothing$ for every copy $K$ of $2^\omega$ in $Z$. It is easy to see that Bernstein sets exist in $\ZFC$, and that they never have the property of Baire (see \cite[Example 8.24]{kechris}). Using Theorem \ref{hurewicz}, one can show that every Bernstein set is $\CB$.

\section{The property of Baire in the restricted sense}

All the results in this section are classical, and they will be needed in the next section. The exposition is based on \cite[Appendix D]{medinit}. Given a space $Z$, we will denote by $\BB(Z)$ be the collection of all subsets of $Z$ that have the property of Baire. For proofs of the following two well-known results, see \cite[Proposition 8.22]{kechris} and \cite[Proposition 8.23]{kechris} respectively.

\begin{proposition}\label{sigmapb}
Let $Z$ be a space. Then $\BB(Z)$ is the smallest $\sigma$-algebra of subsets of $Z$ containing all open sets and all meager sets.
\end{proposition}

\begin{proposition}\label{equivalentpb}
Let $Z$ be a space. Then the following conditions are equivalent for every subset $X$ of $Z$.
\begin{itemize}
\item $X\in\BB(Z)$.
\item $X=G\cup M$, where $G$ is a $\Gd$ subset of $Z$ and $M$ is a meager subset of $Z$.
\end{itemize}
\end{proposition}

Recall that a subset $X$ of a space $Z$ has the \emph{property of Baire in the restricted sense} if $X\cap S\in\BB(S)$ for every subspace $S$ of $Z$ (see \cite[Subsection VI of Section 11]{kuratowski}). We will denote by $\BB_r(Z)$ the collection of subsets of $Z$ that have the property of Baire in the restricted sense. Using Proposition \ref{sigmapb}, it is easy to check that $\BB_r(Z)$ is a $\sigma$-algebra.

The inclusion $\BB_r(Z)\subseteq\BB(Z)$ is obvious. To see that the reverse inclusion need not hold, let $Z=2^\omega\times 2^\omega$ and fix $z\in 2^\omega$. Let $X$ be a Bernstein set in $S=\{z\}\times 2^\omega$. In particular, $X\cap S=X\notin\BB(S)$, so $X\notin\BB_r(Z)$. However, since $X$ is nowhere dense in $Z$, it is clear that $X\in\BB(Z)$. Notice that the same example $X$ shows that, in the following proposition, the hypothesis ``$X\in\BB_r(Z)$'' cannot be weakened to ``$X\in\BB(Z)$''.
\begin{proposition}\label{dichotomy}
Let $Z$ be a Polish space, and assume that $X\in\BB_r(Z)$. Then either $X$ has a dense Polish subspace or $X$ is not Baire.
\end{proposition}
\begin{proof}
Since $X\in\BB(\cl(X))$, by Proposition \ref{equivalentpb}, there exist a $\Gd$ subset $G$ of $\cl(X)$ and a meager subset $M$ of $\cl(X)$ such that $X=M\cup G$. Notice that $G$ is Polish because $\cl(X)$ is Polish. Furthermore, since $X$ is dense in $\cl(X)$, the set $M$ is meager in $X$ as well. Therefore, if $G$ is dense in $X$, then the first alternative will hold. Otherwise, the second alternative will hold.
\end{proof}

Finally, we will point out a significant class of sets that have the property of Baire in the restricted sense. Given a Polish space $Z$, we will denote by $\Aa_\sigma(Z)$ the $\sigma$-algebra of subsets of $Z$ generated by the analytic sets.

\begin{proposition}\label{sigmaanalytic}
Let $Z$ be a Polish space. Then $\Aa_\sigma(Z)\subseteq\BB_r(Z)$.
\end{proposition}
\begin{proof}
Since, as we have already observed, $\BB_r(Z)$ is a $\sigma$-algebra, it will be enough to show that every analytic subset of $Z$ has the property of Baire in the restricted sense. Trivially, every closed subset of $Z$ has the property of Baire in the restricted sense. Therefore, since every analytic set is obtained by applying Souslin operation $\Aa$ to a family of closed sets (see \cite[Theorem 25.7]{kechris}), it will be enough to show that the property of Baire in the restricted sense is preserved by operation $\Aa$. This is a straightforward corollary of the classical fact that the property of Baire is preserved by operation $\Aa$ (see \cite[Corollary 29.14]{kechris}).
\end{proof}

\section{Strengthening a result of Hru\v{s}\'ak and Zamora Avil\'es}

The main result of this section is Theorem \ref{coanalyticcdhpower}, which gives the promised strengthening of Theorem \ref{borelcdhpower} and answers the second part of Question \ref{optimizedefinability}. We will need a few preliminaries. Proposition \ref{notcdh} first appeared as \cite[Proposition 13]{kunenmedinizdomskyy}. Proposition \ref{meagerdenseGdelta} first appeared as \cite[Lemma 3.2]{fitzpatrickzhoub}. Corollary \ref{cdhlambda} first appeared as the first part of \cite[Theorem 3.4]{fitzpatrickzhoub}. Proposition \ref{baire} first appeared as \cite[Theorem 3.1]{hrusakzamoraaviles}.

\begin{proposition}[Kunen, Medini, Zdomskyy]\label{notcdh}
Let $X$ be a space that is not $\CB$ but has a dense $\CB$ subspace. Then $X$ is not $\CDH$.
\end{proposition}
\begin{proof}
Let $D$ be a dense $\CB$ subspace of $X$, and let $A$ be a countable dense subset of $D$. By Theorem \ref{hurewicz}, there exists a closed subspace $Q$ of $X$ that is homeomorphic to $\QQQ$. Extend $Q$ to a countable dense subset $B$ of $X$. Clearly there is no homeomorphism $h:X\longrightarrow X$ such that $h[A]=B$.
\end{proof}

\begin{proposition}[Fitzpatrick, Zhou]\label{meagerdenseGdelta}
Every meager space has a countable dense $\Gd$ subset.
\end{proposition}
\begin{proof}
Let $\{U_n:n\in\omega\}$ be a countable base for $X$. Assume that $X=\bigcup_{\ell\in\omega}K_\ell$, where each $K_\ell$ is a closed nowhere dense subset of $X$. Let $D=\{d_n:n\in\omega\}$, where each $d_n\in U_n\setminus\bigcup_{\ell<n}K_\ell$. It is clear that $D$ is a countable dense subset of $X$. To see that $D$ is $\Gd$, notice that
$$
X\setminus D=\bigcup_{\ell\in\omega}(K_\ell\setminus\{d_n:n\leq\ell\})
$$
is $\Fs$ because each $K_\ell\setminus\{d_n:n\leq\ell\}$ is $\Fs$.
\end{proof}
\begin{corollary}[Fitzpatrick, Zhou]\label{cdhlambda}
Let $X$ be a meager $\CDH$ space. Then $X$ is a $\lambda$-set.
\end{corollary}
\begin{proof}
By Proposition \ref{meagerdenseGdelta}, there exists a countable dense $\Gd$ subset $A$ of $X$. Now let $D$ be an arbitrary countable subset of $X$. Extend $D$ to a countable dense subset $B$ of $X$. Notice that $B$ is $\Gd$ because there exists a homeomorphism $h:X\longrightarrow X$ such that $h[A]=B$. Since $B\setminus D$ is countable, it follows that $D$ is $\Gd$.
\end{proof}

\begin{proposition}[Hru\v{s}\'ak, Zamora Avil\'es]\label{baire}
Let $X$ be a space such that $X^\omega$ is $\CDH$. Then $X$ is Baire.
\end{proposition}
\begin{proof}
If $|X|\leq 1$ then $X$ is obviously Baire, so assume that $|X|\geq 2$. In particular, $X^\omega$ contains a copy of $2^\omega$. Assume, in order to get a contradiction, that $U$ is a non-empty meager open subset of $X$. Let $M_n=\{x\in X^\omega:x(n)\in U\}$ for $n\in\omega$, and observe that each $M_n$ is a meager subset of $X^\omega$. Notice that $X^\omega$ is meager because
$$
X^\omega=(X\setminus U)^\omega\cup\bigcup_{n\in\omega}M_n
$$
and $(X\setminus U)^\omega$ is a closed nowhere dense subset of $X^\omega$. Therefore, $X^\omega$ is a $\lambda$-set by Corollary \ref{cdhlambda}. This contradicts the fact that $X^\omega$ contains a copy of $2^\omega$.
\end{proof}

\begin{theorem}\label{coanalyticcdhpower}
Let $X$ be a coanalytic subspace of $2^\omega$. Then the following are equivalent.
\begin{enumerate}
\item\label{polish} $X$ is Polish.
\item\label{cdh} $X^\omega$ is $\CDH$.
\end{enumerate}
\end{theorem}
\begin{proof}
In order to prove the implication $(\ref{polish})\rightarrow (\ref{cdh})$, assume that $X$ is Polish and that $|X|\geq 2$. Then $X^\omega$ is a crowded zero-dimensional Polish space that is either compact or nowhere locally compact. It follows that $X^\omega\approx 2^\omega$ or $X^\omega\approx\omega^\omega$. In both cases, $X^\omega$ is homogeneous and zero-dimensional, hence $\SLH$. In conclusion, $X^\omega$ is $\CDH$ by Theorem \ref{maincdh}. Notice that Theorem \ref{sufficient} gives an alternative proof of the implication $(\ref{polish})\rightarrow (\ref{cdh})$, since being Polish is obviously stronger than being countably controlled (see Definition \ref{cc}).

In order to prove the implication $(\ref{cdh})\rightarrow (\ref{polish})$, assume that $X^\omega$ is $\CDH$. By Proposition \ref{baire}, it follows that $X$ is Baire. Clearly $X\in\Aa_\sigma(2^\omega)$, so $X\in\BB_r(2^\omega)$ by Proposition \ref{sigmaanalytic}. Therefore, $X$ has a dense Polish subspace by Proposition \ref{dichotomy}. In particular, $X^\omega$ has a dense $\CB$ subspace, hence it is $\CB$ by Proposition \ref{notcdh}. Notice that $X$ is homeomorphic to a closed subspace of $X^\omega$, so it is $\CB$ as well. Since $X$ is coanalytic, it follows that $X$ is Polish by Theorem \ref{hurewicz}.
\end{proof}

\section{Knaster-Reichbach covers}

The results in this section and the next are known and by no means optimal: we simply tried to make the main part of this article as self-contained as possible. Knaster-Reichbach covers were introduced in \cite{knasterreichbach} and have been successfully applied by several authors, including van Engelen, Medvedev and Ostrovski{\u\i}. Let us mention for example the articles \cite{vanengelen}, \cite{medvedevf}, \cite{medvedevb}, \cite{medvedevc}, \cite{medvedevp} and \cite{ostrovskii}, where one can find much more general results than the ones stated here. The first application of this technique to the theory of countable dense homogeneity was recently given by Hern\'andez-Guti\'errez, Hru\v{s}\'ak and van Mill in \cite{hernandezgutierrezhrusakvanmill}.

Fix a homeomorphism $h:E\longrightarrow F$ between closed nowhere dense subsets of $2^\omega$. We will say that $\langle \VV,\WW,\psi\rangle$ is a \emph{Knaster-Reichbach cover} (briefly, a $\KR$-cover) for $\langle 2^\omega\setminus E,2^\omega\setminus F,h\rangle$ if the following conditions hold.
\begin{itemize}
\item $\VV$ is a partition of $2^\omega\setminus E$ consisting of non-empty clopen subsets of $2^\omega$.
\item $\WW$ is a partition of $2^\omega\setminus F$ consisting of non-empty clopen subsets of $2^\omega$.
\item $\psi:\VV\longrightarrow\WW$ is a bijection.
\item If $f:2^\omega\longrightarrow 2^\omega$ is a bijection such that $h\subseteq f$ and $f[V]=\psi(V)$ for every $V\in\VV$, then $f$ is continuous on $E$ and $f^{-1}$ is continuous on $F$.
\end{itemize}
\noindent Whenever $f:2^\omega\longrightarrow 2^\omega$ is a bijection such that $f[V]=\psi(V)$ for every $V\in\VV$, we will say that $f$ \emph{respects} $\psi$.

The following lemma will be the key ingredient at the inductive step in the proof of Theorem \ref{main}. The proof given here is inspired by \cite[Theorem 3.1]{vanmillc}.

\begin{lemma}\label{KRexists}
Let $h:E\longrightarrow F$ be a homeomorphism between closed nowhere dense subsets of $2^\omega$. Then there exists a $\KR$-cover for $\langle 2^\omega\setminus E,2^\omega\setminus F,h\rangle$.
\end{lemma}
\begin{proof}
The case in which $E$ and $F$ are empty is trivial, so assume that $E$ and $F$ are non-empty. Let $X\oplus Y$ be the disjoint topological sum of two spaces that are homeomorphic to $2^\omega$. Without loss of generality, assume that $E$ is a subspace of $X$ and $F$ is a subspace of $Y$. Consider the equivalence relation on $X\oplus Y$ obtained by identifying $x$ with $h(x)$ for every $x\in E$. Denote by $Z$ the corresponding quotient space. For simplicity, we will freely identify an element of $X\oplus Y$ with its equivalence class in $Z$. Notice that $Z$ is separable and metrizable by \cite[Theorem A.11.2]{vanmilli}. Furthermore, it is clear that $Z$ is compact.

\newpage

Fix an admissible metric $\di$ on $Z$. Fix a partition $\VV$ of $X\setminus E$ consisting of non-empty clopen subsets of $X$ and a partition $\WW$ of $Y\setminus F$ consisting of non-empty clopen subsets of $Y$ such that $\diam(V_k)\to 0$ and $\diam(W_k)\to 0$ as $k\to\infty$, where $\VV=\{V_k:k\in\omega\}$ and $\WW=\{W_k:k\in\omega\}$ are injective enumerations. Pick $a_k\in V_k$ and $b_k\in W_k$ for each $k$. It is easy to check that the sequences $\langle a_k:k\in\omega\rangle$ and $\langle b_k:k\in\omega\rangle$ have the same set of limit points in $Z$, namely $E=F$. Therefore, by a result of von Neumann from \cite[pages 11-12]{vonneumann} (see also \cite{halmos} and \cite{yorke} for simpler proofs), there exists a bijection $\pi:\omega\longrightarrow\omega$ such that $\di(a_k,b_{\pi(k)})\to 0$ as $k\to\infty$.

Define $\psi:\VV\longrightarrow\WW$ by setting $\psi(V_k)=W_{\pi(k)}$ for $k\in\omega$. We claim that $\langle \VV,\WW,\psi\rangle$ is a $\KR$-cover for $\langle 2^\omega\setminus E,2^\omega\setminus F,h\rangle$. Let $f:X\longrightarrow Y$ be a bijection that extends $h$ and respects $\psi$. We need to show that $f$ is continuous on $E$ and $f^{-1}$ is continuous on $F$. Since these proofs are similar, we will only deal with the first statement. So fix $x\in E$, and let $\langle x_n:n\in\omega\rangle$ be a sequence that converges to $x$ in $X$. Let $y=f(x)$, and notice that $x=y$ in $Z$. We will show that the sequence $\langle f(x_n):n\in\omega\rangle$ converges to $y$ in $Y$. Fix a neighborhood $W$ of $y$ in $Y$. Let $\varepsilon >0$ be such that $\Bi(y,\varepsilon)\cap Y\subseteq W$, where $\Bi(y,\varepsilon)=\{z\in Z:\di(z,y)<\varepsilon\}$. It will be enough to show that $f(x_n)\in\Bi(y,\varepsilon)$ for all but finitely many values of $n$.

The case in which $x_n\in E$ for all but finitely many values of $n$ is trivial by the continuity of $h$, so assume that $x_n\notin E$ for infinitely many values of $n$. For every $n\in\omega$ such that $x_n\notin E$, define $k_n\in\omega$ to be the unique index such that $x_n\in V_{k_n}$, and notice that $f(x_n)\in W_{\pi(k_n)}$ because $f$ respects $\psi$. Furthermore, it is easy to check that $b_{\pi(k_n)}\to y$ as $n\to\omega$, since $a_{k_n}\to x=y$ and $\di(a_{k_n},b_{\pi(k_n)})\to 0$ as $n\to\omega$. Therefore, given that
$$
\di(f(x_n),y)\leq\di(f(x_n),b_{\pi(k_n)})+\di(b_{\pi(k_n)},y),
$$
there exists $m\in\omega$ such that $f(x_n)\in\Bi(y,\varepsilon)$ whenever $n\geq m$ and $x_n\notin E$. Finally, since $h$ is continuous, we can also assume without loss of generality that $f(x_n)\in\Bi(y,\varepsilon)$ whenever $n\geq m$ and $x_n\in E$. 
\end{proof}

\section{Knaster-Reichbach systems}

Throughout this section, we will denote by $\di$ a fixed admissible metric on $2^\omega$. We will say that a sequence $\langle\langle h_n,\KK_n\rangle:n\in\omega\rangle$ is a \emph{Knaster-Reichbach system} (briefly, a $\KR$-system) if the following conditions are satisfied.
\begin{enumerate}
\item Each $h_n:E_n\longrightarrow F_n$ is a homeomorphism between closed nowhere dense subsets of $2^\omega$.
\item\label{hincreasing} $h_m\subseteq h_n$ whenever $m\leq n$.
\item\label{KRcovercondition} Each $\KK_n=\langle \VV_n,\WW_n,\psi_n\rangle$ is a $\KR$-cover for $\langle 2^\omega\setminus E_n,2^\omega\setminus F_n,h_n\rangle$.
\item\label{mesh} $\mesh(\VV_n)\leq 2^{-n}$ and $\mesh(\WW_n)\leq 2^{-n}$ for each $n$.
\item\label{refinement} $\VV_m$ refines $\VV_n$ and $\WW_m$ refines $\WW_n$ whenever $m\geq n$.
\item\label{coherencepsi} Given $U\in\VV_m$ and $V\in\VV_n$ with $m\geq n$, then $U\subseteq V$ if and only if $\psi_m(U)\subseteq\psi_n(V)$.
\end{enumerate}

\begin{theorem}\label{KRsystem}
Assume that $\langle\langle h_n,\KK_n\rangle:n\in\omega\rangle$ is a $\KR$-system. Then there exists a homeomorphism $h:2^\omega\longrightarrow 2^\omega$ such that $h\supseteq\bigcup_{n\in\omega}h_n$.
\end{theorem}

\begin{proof}
Let $E=\bigcup_{n\in\omega}E_n$ and $F=\bigcup_{n\in\omega}F_n$. Given $x\in 2^\omega\setminus E$ and $n\in\omega$, denote by $V_n^x$ the unique element of $\VV_n$ that contains $x$. Given $y\in 2^\omega\setminus F$ and $n\in\omega$, denote by $W_n^y$ the unique element of $\WW_n$ that contains $y$.

\newpage

If $x\in E_n$ for some $n\in\omega$, define $h(x)=h_n(x)$. The choice of $n$ is irrelevant by condition (\ref{hincreasing}). Now assume that $x\in 2^\omega\setminus E$. Notice that every finite subset of $\{\psi_n(V_n^x):n\in\omega\}$ has non-empty intersection by conditions (\ref{refinement}) and (\ref{coherencepsi}). Since $2^\omega$ is compact and condition (\ref{mesh}) holds, it follows that there exists $y\in 2^\omega$ such that $\bigcap_{n\in\omega}\psi_n(V_n^x)=\{y\}$. Set $h(x)=y$. This concludes the definition of $h$.

Similarly, define $g:2^\omega\longrightarrow 2^\omega$ by setting $g(y)=h_n^{-1}(y)$ if $y\in F_n$ for some $n\in\omega$, and $g(y)=x$ if $y\in 2^\omega\setminus F$, where $x\in 2^\omega$ is such that $\bigcap_{n\in\omega}\psi_n^{-1}(W_n^y)=\{x\}$. It is easy to check that $g=h^{-1}$, hence $h$ is a bijection.

It is straightforward to verify that $h$ respects $\psi_n$ for each $n$. Therefore, by condition (\ref{KRcovercondition}), $h$ is continuous on $E$ and $h^{-1}$ is continuous on $F$. It remains to show that $h$ is continuous on $2^\omega\setminus E$ and that $h^{-1}$ is continuous on $2^\omega\setminus F$. Since these proofs are similar, we will only deal with the first statement. Fix $x\in 2^\omega\setminus E$, and let $y=h(x)$. Fix a neighborhood $W$ of $y$ in $2^\omega$. By condition (\ref{mesh}), there exists $n\in\omega$ such that $W_n^y\subseteq W$. It remains to observe that $h[V_n^x]=W_n^y$.
\end{proof}

\begin{corollary}\label{KRcor}
Let $X$ be a subspace of $2^\omega$. Assume that $\langle\langle h_n,\KK_n\rangle:n\in\omega\rangle$ is a $\KR$-system satisfying the following additional conditions.
\begin{enumerate}
\item[(7)] $2^\omega\setminus\bigcup_{n\in\omega}E_n\subseteq X$.
\item[(8)] $2^\omega\setminus\bigcup_{n\in\omega}F_n\subseteq X$.
\item[(9)] $h_n[X\cap E_n]=X\cap F_n$ for each $n$.
\end{enumerate}
Then there exists a homeomorphism $h:2^\omega\longrightarrow 2^\omega$ such that $h\supseteq\bigcup_{n\in\omega}h_n$ and $h[X]=X$.
\end{corollary}
\begin{proof}
By Theorem \ref{KRsystem}, there exists a homeomorphism $h:2^\omega\longrightarrow 2^\omega$ such that $h\supseteq\bigcup_{n\in\omega}h_n$. In order to show that $h[X]\subseteq X$, fix $x\in X$. If $x\in\bigcup_{n\in\omega}E_n$, then $h(x)\in X$ by condition (9). On the other hand, if $x\in 2^\omega\setminus\bigcup_{n\in\omega}E_n$ then $h(x)\in 2^\omega\setminus\bigcup_{n\in\omega}F_n$, which implies $h(x)\in X$ by condition (8). A similar argument shows that $h^{-1}[X]\subseteq X$. It follows that $h[X]=X$.
\end{proof}

\section{The main result}

The following two definitions are crucial for our purposes. Recall that a \emph{$\pi$-base} for a space $Z$ is a collection $\BB$ consisting of non-empty open subsets of $Z$ such that for every non-empty open subset $U$ of $Z$ there exists $V\in\BB$ such that $V\subseteq U$.

\begin{definition}
Let $X$ be a subspace of $Z$. We will say that $X$ is \emph{h-homogeneously embedded} in $Z$ if there exists a $\pi$-base $\BB$ for $Z$ consisting of clopen sets and homeomorphisms $\varphi_U:Z\longrightarrow U$ for $U\in\BB$ such that $\varphi_U[X]=X\cap U$.
\end{definition}

\begin{definition}\label{cc}
We will say that a space $X$ is \emph{countably controlled} if for every countable $D\subseteq X$ there exists a Polish subspace $G$ of $X$ such that $D\subseteq G\subseteq X$.
\end{definition}

The technique used in the proof of the following theorem is essentially due to Medvedev (see \cite[Theorem 5]{medvedevp}).
\begin{theorem}\label{main}
Assume that $X$ is h-homogeneously embedded in $2^\omega$ and countably controlled. Then $X$ is $\CDH$. 
\end{theorem}
\begin{proof}
If $X$ is empty then $X$ is obviously $\CDH$, so assume that $X$ is non-empty. Since $X$ is h-homogeneously embedded in $2^\omega$, there exists a (countable) $\pi$-base $\BB$ for $2^\omega$ consisting of clopen sets and homeomorphisms $\varphi_U:2^\omega\longrightarrow U$ for $U\in\BB$ such that $\varphi_U[X]=X\cap U$. In particular, $X$ is dense in $2^\omega$.

Fix a pair $(A,B)$ of countable dense subsets of $X$. Let $D_0=A\cup B$, and given $D_n$ for some $n\in\omega$, define
$$
D_{n+1}=\bigcup\{\varphi_U^{-1}[D_n\cap U]:U\in\BB\}.
$$
In the end, let $D=\bigcup_{n\in\omega}D_n$. It is easy to check that $D$ is a countable dense subset of $2^\omega$ such that $A\cup B\subseteq D\subseteq X$. Furthermore, it is clear that $\varphi_U^{-1}(x)\in D$ whenever $x\in D$ and $U\in\BB$ is such that $x\in U$.

Since $X$ is countably controlled, it is possible to find a $\Gd$ subset $G$ of $2^\omega$ such that $D\subseteq G\subseteq X$. By removing countably many points from $G$, we can assume without loss of generality that $2^\omega\setminus G$ is dense in $2^\omega$. Fix closed nowhere dense subsets $K_\ell$ of $2^\omega$ for $\ell\in\omega$ such that $2^\omega\setminus G=\bigcup_{\ell\in\omega}K_\ell$. Also fix the following injective enumerations.
\begin{itemize}
\item $A=\{a_i:i\in\omega\}$.
\item $B=\{b_j:j\in\omega\}$.
\end{itemize}

Fix an admissible metric $\di$ on $2^\omega$ such that $\diam(2^\omega)\leq 1$. Our strategy is to construct a suitable $\KR$-system $\langle\langle h_n,\KK_n\rangle:n\in\omega\rangle$, then apply Corollary \ref{KRcor} to get a homeomorphism $h:2^\omega\longrightarrow 2^\omega$ such that $h\supseteq\bigcup_{n\in\omega}h_n$ and $h[X]=X$. We will use the same notation as in Section 6. In particular, $h_n:E_n\longrightarrow F_n$ and $\KK_n=\langle\VV_n,\WW_n,\psi_n\rangle$ for each $n$.

Of course, we will have to make sure that conditions (1)-(6) in the definition of a $\KR$-system are satisfied. Furthermore, we will make sure that the following additional conditions are satisfied for every $n\in\omega$.
\begin{enumerate}
\item[(I)] $\bigcup_{\ell<n}K_\ell\subseteq E_n$.
\item[(II)] $\bigcup_{\ell<n}K_\ell\subseteq F_n$.
\item[(III)] $h_n[X\cap E_n]=X\cap F_n$.
\item[(IV)] $\{a_i:i<n\}\subseteq E_n$.
\item[(V)] $\{b_j:j<n\}\subseteq F_n$.
\item[(VI)] $h_n[A\cap E_n]=B\cap F_n$.
\end{enumerate}
Conditions (I)-(III) will guarantee that conditions (7)-(9) in Corollary \ref{KRcor} hold. On the other hand, conditions (IV)-(VI) will guarantee that $h[A]=B$.

Start by letting $h_0=\varnothing$ and $\KK_0=\langle\{2^\omega\},\{2^\omega\},\{\langle 2^\omega,2^\omega\rangle\}\rangle$. Now assume that $\langle h_n,\KK_n\rangle$ is given. First, for any given $V\in\VV_n$, we will define a homeomorphism $h_V:E_V\longrightarrow F_V$, where $E_V$ will be a closed nowhere dense subset of $V$ and $F_V$ will be a closed nowhere dense subset of $\psi_n(V)$. So fix $V\in\VV_n$, and let $W=\psi_n(V)$.

Define the following indices.
\begin{itemize}
\item $\ell(V)=\mini\{\ell\in\omega:K_\ell\cap V\neq\varnothing\}$.
\item $\ell(W)=\mini\{\ell\in\omega:K_\ell\cap W\neq\varnothing\}$.
\item $i(V)=\mini\{i\in\omega:a_i\in V\setminus K_{\ell(V)}\}$.
\item $j(W)=\mini\{j\in\omega:b_j\in W\setminus K_{\ell(W)}\}$.
\end{itemize}
Notice that the indices $\ell(V)$ and $\ell(W)$ are well-defined because $\bigcup_{\ell\in\omega}K_\ell=2^\omega\setminus G$ is dense in $2^\omega$.

Let $S=(V\cap K_{\ell(V)})$. Since $K_{\ell(V)}$ is a closed nowhere dense subset of $2^\omega$, we can fix $U(S)\in\BB$ such that $U(S)\subseteq V\setminus (S\cup\{a_{i(V)}\})$. Let $T=(W\cap K_{\ell(W)})$. Since $K_{\ell(W)}$ is a closed nowhere dense subset of $2^\omega$, we can fix $U(T)\in\BB$ such that $U(T)\subseteq W\setminus (T\cup\{b_{j(W)}\})$.

\newpage

Define $E_V=\{a_{i(V)}\}\cup S\cup\varphi_{U(S)}[T]$ and $F_V=\{b_{j(W)}\}\cup T\cup\varphi_{U(T)}[S]$. Observe that $E_V$ is a closed nowhere dense subset of $V$ and $F_V$ is a closed nowhere dense subset of $W$. Define $h_V:E_V\longrightarrow F_V$ by setting
$$
\left.
\begin{array}{lcl}
& & h_V(x)= \left\{
\begin{array}{ll}
b_{j(W)} & \textrm{if }x=a_{i(V)},\\
\varphi_{U(T)}(x) & \textrm{if }x\in S,\\
(\varphi_{U(S)})^{-1}(x) & \textrm{if }x\in\varphi_{U(S)}[T].
\end{array}
\right.
\end{array}
\right.
$$
It is clear that $h_V$ is a homeomorphism. Therefore, by Lemma \ref{KRexists}, there exists a $\KR$-cover $\langle\VV_V,\WW_V,\psi_V\rangle$ for $\langle V\setminus E_V,W\setminus F_V,h_V\rangle$. Furthermore, it is easy to realize that $h_V[X\cap E_V]=X\cap F_V$, which will allow us to mantain condition (III).

Notice that $\phi_{U(S)}[T]\cap D=\varnothing$, because $\phi_{U}[K_\ell]\cap D=\varnothing$ for every $U\in\BB$ and $\ell\in\omega$ by the choice of $D$. Similarly, one sees that $\phi_{U(T)}[S]\cap D=\varnothing$. Since $A\cup B\subseteq D$, it follows that $h_V[A\cap E_V]=h_V[\{a_{i(V)}\}]=\{b_{j(W)}\}=B\cap F_V$, which will allow us to mantain condition (VI).

Repeat this construction for every $V\in\VV_n$, then let $E_{n+1}=E_n\cup\bigcup\{E_V:V\in\VV_n\}$ and $F_{n+1}=F_n\cup\bigcup\{F_V:V\in\VV_n\}$. Define
$$
h_{n+1}=h_n\cup\bigcup_{V\in\VV_n}h_V,
$$
and observe that $h_{n+1}:E_{n+1}\longrightarrow F_{n+1}$ is a bijection. Now extend $h_V$ to a bijection $f_V:V\longrightarrow\psi_n(V)$ for every $V\in\VV_n$, and let $f_n=h_n\cup\bigcup_{V\in\VV_n}f_V$. Clearly, $f_n:2^\omega\longrightarrow 2^\omega$ is a bijection that extends $h_{n+1}\supseteq h_n$ and respects $\psi_n$. Since $\KK_n=\langle\VV_n,\WW_n,\psi_n\rangle$ is a $\KR$-cover for $\langle 2^\omega\setminus E_n, 2^\omega\setminus F_n,h_n\rangle$, it follows that $h_{n+1}$ is continuous on $E_n$ and $h_{n+1}^{-1}$ is continuous on $F_n$. On the other hand, it is straightforward to check that $h_{n+1}$ is continuous on $E_{n+1}\setminus E_n=\bigcup\{E_V:V\in\VV_n\}$ and $h_{n+1}^{-1}$ is continuous on $F_{n+1}\setminus F_n=\bigcup\{F_V:V\in\VV_n\}$. In conclusion, $h_{n+1}$ is a homeomorphism.

Finally, we define $\KK_{n+1}=\langle\VV_{n+1},\WW_{n+1},\psi_{n+1}\rangle$. Let $\VV_{n+1}=\bigcup\{\VV_V:V\in\VV_n\}$ and $\WW_{n+1}=\bigcup\{\WW_V:V\in\VV_n\}$. By further refining $\VV_{n+1}$ and $\WW_{n+1}$, we can assume that $\mesh(\VV_{n+1})\leq 2^{-(n+1)}$ and $\mesh(\WW_{n+1})\leq 2^{-(n+1)}$. Let $\psi_{n+1}=\bigcup_{V\in\VV_n}\psi_V$. Using the fact that $\langle\VV_V,\WW_V,\psi_V\rangle$ is a $\KR$-cover for $\langle V\setminus E_V,W\setminus F_V,h_V\rangle$ for each $V\in\VV_n$ together with condition (\ref{KRcovercondition}), it is easy to realize that $\KK_{n+1}$ is a $\KR$-cover for $\langle 2^\omega\setminus E_{n+1}, 2^\omega\setminus F_{n+1},h_{n+1}\rangle$.
\end{proof}

\section{Infinite powers and $\lambda'$-sets}

The main result of this section is Theorem \ref{mainexample}, which simultaneously answers Question \ref{Gdeltaquestion}, the first part of Question \ref{optimizedefinability}, and shows that Theorem \ref{coanalyticcdhpower} is sharp. The idea of looking at (the complements of) $\lambda'$-sets is inspired by a recent article of Hern\'andez-Guti\'errez, Hru\v{s}\'ak, and van Mill (more precisely, by \cite[Theorem 4.5]{hernandezgutierrezhrusakvanmill}).

We will need a few preliminary results. The straightforward proofs of the following two propositions are left to the reader.

\begin{proposition}\label{prodpreservehhomemb}
Let $I$ be a countable set. If $X_i$ is h-homogeneously embedded in $Z_i$ for every $i\in I$ then $\prod_{i\in I}X_i$ is h-homogeneously embedded in $\prod_{i\in I}Z_i$.
\end{proposition}

\begin{proposition}\label{prodpreservecc}
Let $I$ be a countable set. If $X_i$ is countably controlled for each $i\in I$ then $\prod_{i\in I}X_i$ is countably controlled.
\end{proposition}

\begin{proposition}\label{existslambda}
There exists a $\lambda'$-set of size $\omega_1$ which is h-homogeneously embedded in $2^\omega$.
\end{proposition}
\begin{proof}
Fix a (countable) $\pi$-base $\BB$ for $2^\omega$ consisting of clopen sets and homeomorphisms $\varphi_U:2^\omega\longrightarrow U$ for $U\in\BB$. Let $X_0$ be a $\lambda'$-set of size $\omega_1$ (whose existence is guaranteed by Theorem \ref{zfclambda}) and, given $X_n$ for some $n\in\omega$, define
$$
X_{n+1}=\bigcup\{\varphi_U[X_n]:U\in\BB\}\cup\bigcup\{\varphi_U^{-1}[X_n\cap U]:U\in\BB\}.
$$
In the end, let $X=\bigcup_{n\in\omega}X_n$. Using induction and Lemma \ref{unionlambda}, it is easy to see that each $X_n$ is a $\lambda'$-set of size $\omega_1$. Therefore, $X$ is a $\lambda'$-set of size $\omega_1$. Finally, the construction of $X$ ensures that $\varphi_U[X]=X\cap U$ for every $U\in\BB$.
\end{proof}

\begin{theorem}\label{mainexample}
There exists a subspace $X$ of $2^\omega$ with the following properties.
\begin{itemize}
\item $X$ is not Polish.
\item $X^\omega$ is $\CDH$.
\item If $\MA+\neg\CH+\omega_1=\omega_1^\Ell$ holds then $X$ is analytic. 
\end{itemize}
\end{theorem}
\begin{proof}
By Proposition \ref{existslambda}, we can fix a $\lambda'$-set $Y$ of size $\omega_1$ which is h-homogeneously embedded in $2^\omega$. Let $X=2^\omega\setminus Y$. By Theorem \ref{martinsolovay}, if $\MA+\neg\CH+\omega_1=\omega_1^\Ell$ holds then $X$ is analytic. It is straightforward to verify that $X$ is is h-homogeneously embedded in $2^\omega$. By Proposition \ref{prodpreservehhomemb}, it follows that $X^\omega$ is h-homogeneously embedded in $(2^\omega)^\omega\approx 2^\omega$. Furthermore, the definition of $\lambda'$-set immediately implies that $X$ is countably controlled. By Proposition \ref{prodpreservecc}, it follows that $X^\omega$ is countably controlled. In conclusion, $X^\omega$ is $\CDH$ by Theorem \ref{main}.

Assume, in order to get a contradiction, that $X$ is Polish. This means that $X$ is a $\Gd$ subspace of $2^\omega$, so $Y$ is an $\Fs$. Since $Y$ is uncountable, it follows that $Y$ contains a copy of $2^\omega$, which contradicts the fact that $Y$ is a $\lambda$-set.
\end{proof}

Observe that, by the remark that follows Theorem \ref{borelcdhpower}, the analytic counterexample given by Theorem \ref{mainexample} could not have been constructed in $\ZFC$.

The following is a classical result (see \cite[Theorem 23.3]{millerd}). For a new, topological proof, based on a result of Baldwin and Beaudoin, see \cite[Theorem 8.1]{medinizdomskyy}.
\begin{theorem}[Martin, Solovay]\label{martinsolovay}
Assume $\MA + \neg\CH + \omega_1=\omega_1^\Ell$. Then every subspace of $2^\omega$ of size $\omega_1$ is coanalytic.
\end{theorem}

\section{A sufficient condition}

The main result of this section is Theorem \ref{sufficient}, which shows that being countably controlled is by itself a sufficient condition on a zero-dimensional space $X$ for the countable dense homogeneity of $X^\omega$. It is easy to realize that Theorem \ref{mainexample} could have been proved using Corollary \ref{everylambda}. However, since the proof of Theorem \ref{sufficient} relies on deep results such as \cite[Theorem 1]{dowpearl} and Theorem \ref{hhompower}, we preferred to make the rest of the paper more self-contained.

The following result is inspired by \cite[Proposition 24]{medinip}, where the proof of the equivalence $(\ref{isolated})\leftrightarrow (\ref{hhom})$ first appeared. Recall that a space $X$ is \emph{h-homogeneous} (or \emph{strongly homogeneous}) if $C\approx X$ for every non-empty clopen subspace $C$ of $X$.
\begin{proposition}\label{strengthenhhom}
Let $X$ be zero-dimensional space such that $|X|\geq 2$. Then the following are equivalent.
\begin{enumerate}
\item\label{isolated} $X^\omega\approx Y^\omega$ for some space $Y$ with at least one isolated point.
\item\label{hhomemb} $X^\omega$ can be h-homogeneously embedded in $2^\omega$.
\item\label{hhom} $X^\omega$ is h-homogeneous.
\end{enumerate}
\end{proposition}
\begin{proof}
In order to prove the implication $(\ref{isolated})\rightarrow (\ref{hhomemb})$, assume that $X^\omega\approx Y^\omega$, where $Y$ is a space with at least one isolated point. Assume without loss of generality that $Y$ is a subspace of $2^\omega$, and let $z\in 2^\omega$ be an isolated point of $Y$. Let $K=\cl(Y)$, where the closure is taken in $2^\omega$, and notice that $z$ remains isolated in $K$. Also notice that $K^\omega$ is crowded because $|X|\geq 2$ and $Y^\omega\approx X^\omega$. It follows that $K^\omega\approx 2^\omega$, so it will be enough to show that $Y^\omega$ is h-homogeneously embedded in $K^\omega$.

Let $[\omega]^{<\omega}=\{F\subseteq\omega:F\text{ is finite}\}$. Given any $F\in [\omega]^{<\omega}$, define
$$
U_F=\{x\in K^\omega: x(n)=z\text{ for all }n\in F\},
$$
and notice that each $U_F$ is a clopen subset of $K^\omega$. Furthermore, it is clear that $\{U_F:F\in [\omega]^{<\omega}\}$ is a local base for $K^\omega$ at $\langle z,z,\ldots\rangle$. By \cite[Theorem 1]{dowpearl}, given any $x\in Y^\omega$, there exists a homeomorphism $h_x:K^\omega\longrightarrow K^\omega$ such that $h_x[Y^\omega]=Y^\omega$ and $h_x(\langle z,z,\ldots\rangle)=x$. Fix a countable dense subset $D$ of $Y^\omega$. It is easy to realize that the collection
$$
\BB=\{h_x[U_F]:x\in D, F\in [\omega]^{<\omega}\}
$$
is a countable $\pi$-base for $K^\omega$ consisting of clopen sets.

For every $F\in [\omega]^{<\omega}$, fix a bijection $\pi_F:\omega\setminus F\longrightarrow\omega$, then define $h_F:K^\omega\longrightarrow U_F$ by setting
$$
\left.
\begin{array}{lcl}
& & h_F(x)(n)= \left\{
\begin{array}{ll}
z & \textrm{if }n\in F,\\
x(\pi_F(n)) & \textrm{if }n\in\omega\setminus F
\end{array}
\right.
\end{array}
\right.
$$
for every $x\in K^\omega$ and $n\in\omega$. One can easily check that each $h_F$ is a homeomorphism such that $h_F[Y^\omega]=Y^\omega\cap U_F$. Given any $U\in\BB$, where $U=h_x[U_F]$ for some $x\in D$ and $F\in [\omega]^{<\omega}$, let $\varphi_U=h_x\circ h_F$. It is straightforward to verify that each $\varphi_U:K^\omega\longrightarrow U$ is a homeomorphism such that $\varphi_U[Y^\omega]=Y^\omega\cap U$.

In order to prove the implication $(\ref{hhomemb})\rightarrow (\ref{hhom})$, assume that $X^\omega$ is h-homogeneously embedded in $2^\omega$. In particular, $X^\omega$ has a $\pi$-base consisting of clopen sets that are homeomorphic to $X^\omega$. If $X^\omega$ is compact then $X^\omega\approx 2^\omega$, which is well-known to be h-homogeneous. On the other hand, if $X^\omega$ is non-compact then it is non-pseudocompact (see \cite[Proposition 3.10.21 and Theorem 4.1.17]{engelking}), in which case the desired result follows from a theorem of Terada (see \cite[Theorem 2.4]{terada} or \cite[Theorem 2 and Appendix A]{medinip}).

In order to prove the implication $(\ref{hhom})\rightarrow (\ref{isolated})$, assume that $X^\omega$ is h-homogeneous. It will be enough to show that $X^\omega$ and $(X\oplus 1)^\omega$ are both homeomorphic to the space $C=(X\oplus 1)^\omega\times X^\omega$, where $X\oplus 1$ denotes the space obtained by adding one isolated point to $X$. Notice that $X^\omega$ can be partitioned into two non-empty clopen subsets because $|X|\geq 2$. Therefore 
$$
X^\omega\approx X^\omega\oplus X^\omega\approx (X\times X^\omega )\oplus X^\omega\approx (X\oplus 1)\times X^\omega.
$$
By taking the $\omega$-th power of both sides, one sees that $X^\omega\approx C$. On the other hand, we know that $(X\oplus 1)^\omega$ is h-homogeneous by the implication $(\ref{isolated})\rightarrow (\ref{hhom})$. Since
$$
(X\oplus 1)^\omega\approx (X\oplus 1)\times (X\oplus 1)^\omega\approx (X\times (X\oplus 1)^\omega)\oplus (X\oplus 1)^\omega,
$$
it follows that $(X\oplus 1)^\omega\approx X\times (X\oplus 1)^\omega$. By taking the $\omega$-th power of both sides, one sees that $(X\oplus 1)^\omega\approx C$.
\end{proof}

The following result has been obtained independently by van Engelen (see \cite[Theorem 4.4]{vanengelen}) and Medvedev (see \cite[Corollary 6]{medvedevb}).

\begin{theorem}[van Engelen; Medvedev]\label{hhompower}
Let $X$ be a zero-dimensional space. If $X$ has a dense Polish subspace then $X^\omega$ is h-homogeneous.
\end{theorem}
\begin{corollary}\label{densePolishcor}
Let $X$ be a zero-dimensional space such that $|X|\geq 2$. If $X$ has a dense Polish subspace then $X^\omega$ can be h-homogeneously embedded in $2^\omega$.
\end{corollary}
\begin{proof}
Apply Proposition \ref{strengthenhhom}.
\end{proof}

\begin{theorem}\label{sufficient}
Let $X$ be a zero-dimensional countably controlled space. Then $X^\omega$ is $\CDH$.
\end{theorem}
\begin{proof}
The case $|X|=1$ is trivial, so assume that $|X|\geq 2$. Clearly, the fact that $X$ is countably controlled implies that $X$ has a Polish dense subspace. Therefore $X^\omega$ can be h-homogeneously embedded in $2^\omega$ by Corollary \ref{densePolishcor}. Furthermore, Proposition \ref{prodpreservecc} shows that $X^\omega$ is countably controlled. In conclusion, $X^\omega$ is $\CDH$ by Theorem \ref{main}.
\end{proof}
\begin{corollary}\label{everylambda}
If $Y$ is a $\lambda'$-set then $(2^\omega\setminus Y)^\omega$ is $\CDH$.
\end{corollary}

It seems natural to wonder whether, in the above theorem, it would be enough to assume that $X$ has a dense Polish subspace, instead of assuming that $X$ is countably controlled. The following simple proposition shows that this is not the case.

\begin{proposition}
There exists a zero-dimensional space $X$ such that $X$ has a dense Polish subspace while $X^\omega$ is not $\CDH$.
\end{proposition}
\begin{proof}
Fix $z\in 2^\omega$. Let $D=2^\omega\times (2^\omega\setminus\{z\})$, and fix a countable dense subset $Q$ of $2^\omega\times\{z\}$. Define
$$
X=Q\cup D\subseteq 2^\omega\times 2^\omega.
$$
It is clear that $D$ is a dense Polish subspace of $X$. Furthermore, $X$ is not Polish because $Q$ is a closed countable crowded subspace of $X$. Since $X$ is a coanalytic subspace of $2^\omega\times 2^\omega\approx 2^\omega$ (actually, it is $\sigma$-compact), if follows that $X^\omega$ is not $\CDH$ by Theorem \ref{coanalyticcdhpower}.
\end{proof}

Finally, we remark that, by Theorem \ref{consistentanswer}, it is not possible to prove in $\ZFC$ that being countably controlled (or even having a dense Polish subspace) is a necessary condition for the countable dense homogeneity of $X^\omega$.

\end{document}